\documentclass{ws-jktr}

\begin{document}

\markboth{D.A. Fedoseev, V.O. Manturov, Z. Cheng}
{On marked braid groups}

\catchline{}{}{}{}{}

\title{ON MARKED BRAID GROUPS}

\author{DENIS A. FEDOSEEV}

\address{Moscow State University, Chair of Differential Geometry and Applications}

\author{VASSILY O. MANTUROV}

\address{Bauman Moscow State Technical University, Chair FN--12}

\author{ZHIYUN CHENG}

\address{School of Mathematical Sciences, Beijing Normal University
\\Laboratory of Mathematics and Complex Systems, Ministry of
Education, Beijing 100875, China}
\maketitle

\begin{abstract}
In the present paper, we introduce $\mathbb{Z}_2$-braids and, more generally, $G$-braids for an arbitrary group $G$.
They form a natural group-theoretic counterpart of $G$-knots, see \cite{reidmoves}. The underlying idea, used in the construction of these objects --- decoration of crossings with some additional information --- generalizes an important notion of {\it parity} introduced by the second author (see \cite{parity}) to different combinatorically--geometric theories, such as knot theory, braid theory and others.
These objects act as natural enhancements of classical (Artin) braid groups.

The notion of dotted braid group is introduced:
classical (Artin) braid groups live inside dotted braid groups as those elements having presentation with no dots
on the strands.

The paper is concluded by a list of unsolved problems.
\end{abstract}

\keywords{Braid, virtual braid, group, presentation, parity.}

\ccode{Mathematics Subject Classification 2000: 57M25, 57M27}

\section{Introduction}	
In the present paper we introduce the notion of $\mathbb{Z}_2-$braids as well as their generalization: groups of $G-$braids for any arbitrary group $G$. They form a natural group--theoretical analog of $G-$knots, first introduced by the second author in \cite{reidmoves}. Those objects naturally generalize the classical (Artin) braids. $G-$braids have a natural structure: every classical crossing of a braid is decorated with an element of $G$.

We also give another presentation of the group of $\mathbb{Z}_2-$braids (and even of a broader group) where the information about the group elements is contained in the braid strands. That leads one to the notion of dotted braids: classical (Artin) braids live inside the dotted braids group, being braids with a diagram with no dots on the strands. A similar object is defined: twisted dotted braids.

In theories dealing with additional structures on objects' diagrams (say, parity, dots, group structures) an important role is played by the theorems of the following form: {\it if two objects are equivalent as objects with an additional structure (hence with a broader set of moves allowed) then they are equivalent as objects without this structure}. An example of this kind of theorems is given by the theorem on classical braids being equivalent as virtual yielding their equivalence as classical (see, for example, \cite{braids}).

There are two theorems of that type in the present paper. The first one deals with classical braids being equivalent as $\mathbb{Z}_2-$braids and the second --- with $\mathbb{Z}_2-$braids equivalent as dotted braids. In both cases the equivalence can be ``descended'' to a smaller class of braids with a smaller set of allowed moves.

Virtual braids as well as virtual knots possess many nice properties which lead to powerful invariants (see \cite{bracket}).
As it has been recently found (see \cite{nonreid}),
picture--valued invariants appear in the case of classical braids by using groups $G_{n}^{3}$.

Thus, a natural question arises: are there any models representing virtual braid groups inside classical
braid groups? Certainly, if we do not admit further modifications, the answer will be negative
since classical braid groups have no torsion.

The present paper is a first step of understanding the above mentioned question by introducing
the idea of parity into the study of classical braids.

A list of unsolved problems is given in the present paper.

\section{Basic definitions}
First let us recall the definition of a \textit{$n-$stand braid}. There exists several (equivalent) definitions; we will formulate the one we will be using in the paper.

\begin{definition}
The $n-$strand braids group is a group given by $(n-1)$ generators $\sigma_1, \dots, \sigma_{n-1}$ and the following set of relations:
\begin{equation}\label{far_comm}
\sigma_i \sigma_j=\sigma_j\sigma_i
\end{equation}
for all $|i-j|\ge 2$ ({\rm far commutativity}) and
\begin{equation}\label{raid3}
\sigma_i \sigma_{i+1}\sigma_i=\sigma_{i+1}\sigma_i\sigma_{i+1}
\end{equation}
for $1\le i\le n-2$.
\end{definition}
This set of relations is called the \textit{Artin relations}.

In other words, the braid group elements are the equivalence classes of words in the alphabet $\sigma_i, \sigma_i^{-1}$ modulo the above--mentioned relations. Such words are also called \textit{braid--words}.

The definition of braids via the Artin presentation can easily be interpreted geometrically. To this end the notion of a \textit{flat diagram} of a braid is introduced.

\begin{definition}
{\rm Flat diagram of a $n-$strand braid} is a graph inside a rectangle $\mathbb{R}\times[0,1]$ with the following additional structure and properties:
\begin{enumerate}
\item The graph vertices of valency 1 are the points $[i,0]$ and $[i,1]$, $i=1,\dots, n$.
\item All other vertices are of valency four. The opposite edges in such vertices are at the angle of 180 degrees.
\item The braid strands monotonously go down with respect to the coordinate $t$.
\item Every four--valent vertex is decorated with a ``overcrossing--undercrossing'' structure.
\end{enumerate}
\end{definition}

Two diagrams are considered equal if they can be connected by a chain of flat isotopies (with the condition of the strands always being monotonously going down), second and third Reidemeister moves.
\begin{figure}
\begin{center}
\includegraphics[scale=0.7]{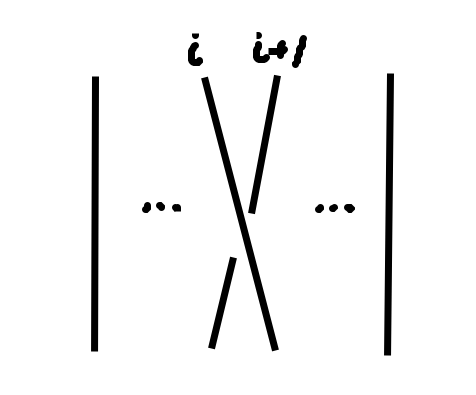}
\caption{The generator $\sigma_i$.}
\label{generators}
\end{center}
\end{figure}
Now, for the generators $\sigma_i$ of the braid group take the braids of the following form: $\sigma_i$ for $i=1, \dots, n-1$ consists of $n-2$ intervals connecting the points $[k,1]$ with $[k,0]$ for all $k\neq i,k \neq i+1$ and two intervals $[i,0]-[i+1,1],\, [i+1,0]-[i,1]$ such that the latter goes over the former (see Fig. \ref{generators}). For the inverse element $\sigma_i^{-1}$ we take the same braid but with the crossing structure in the only crossing reversed (see Fig. \ref{inverse}). It can be easily seen that the existence of the inverse is exactly the invariance under the second Reidemeister move, the relation (\ref{raid3}) --- the third Reidemeister move, and the far commutativity can be understood as the commutativity of the crossings with no mutual strands.
\begin{figure}[h!]
\begin{center}
\includegraphics[scale=0.7]{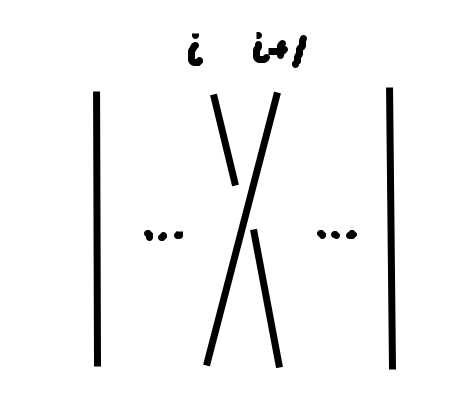}
\caption{The inverse to a generator $\sigma_i$.}
\label{inverse}
\end{center}
\end{figure}
\section{$\mathbb{Z}_2-$braids}

Now let us define a new object: $\mathbb{Z}_2-$briads or \textit{braids with parity} (see \cite{reidmoves}). From now on we will use the algebraic definition but keep in mind that there is always an underlying geometric object, probably with an additional structure.

\begin{definition}[$\mathbb{Z}_2-$braid group]
The $n-$strand $\mathbb{Z}_2-$braid group is a group with generators $\sigma_{i,0}, \sigma_{i,1},$ $i=1,\dots, n-1$ and the following set of relations:
\begin{equation}\label{far_comm_z}
\sigma_{i,\varepsilon} \sigma_{j,\eta}=\sigma_{j,\eta}\sigma_{i,\varepsilon}
\end{equation}
for all $|i-j|\ge 2$ and all $\varepsilon, \eta \in \{0,1\}$ ({\rm far commutativity}) and
\begin{equation}\label{raid3_z}
\sigma_{i,\varepsilon} \sigma_{i+1,\eta}\sigma_{i,\xi}=\sigma_{i+1,\xi}\sigma_{i,\eta}\sigma_{i+1,\varepsilon}
\end{equation}
where $1\le i\le n-2$ and $\varepsilon+\eta+\xi\equiv 0 \mod 2$.
\end{definition}

We will call the $\sigma_{i,0}$ generators \textit{even} and the $\sigma_{i,1}$ ones --- \textit{odd}. The $n-$strand $\mathbb{Z}_2-$braids group will be denoted by $Br_2^n$.

From the geometric point of view this definition means an additional structure (``parity'') being introduced in the braid crossings. The parity is such that every two crossings in a second Reidemeister move are both even or both odd and among the three crossing in the third move either two or zero are odd.

When working with objects with additional structure which are related to objects with no such structure one should keep in mind the following important principle: {\it if two objects are equivalent as objects with additional structure (and thus with a broader set of moves and equivalences allowed) they are equivalent as objects without the structure (therefore connected with a chain of moves from a smaller set)}. In case of classical and $\mathbb{Z}_2-$braids this principle is true. In other words, the following theorem holds:

\begin{theorem} \label{classic_in_2}
If two braids without odd crossings are equivalent as $\mathbb{Z}_2-$braids, they are equivalent in the class of the braids with no odd crossings ({\rm even braids}).
\end{theorem}

To prove this fact we need a lemma relating $\mathbb{Z}_2-$braids with virtual braids. Let us recall the definition of virtual braids:

\begin{definition}
The $n-$strand virtual braids group is a group with $(n-1)$ generators $\sigma_1, \dots, \sigma_{n-1}$ and $(n-1)$ generators $\zeta_1, \dots, \zeta_{n-1}$ and the following set of relations:
\begin{equation}\label{far_comm_v1}
\sigma_i \sigma_j=\sigma_j\sigma_i
\end{equation}
for all $|i-j|\ge 2$ ({\rm far commutativity}) and
\begin{equation}\label{raid3_v1}
\sigma_i \sigma_{i+1}\sigma_i=\sigma_{i+1}\sigma_i\sigma_{i+1}
\end{equation}
for $1\le i\le n-2$;
\begin{equation}\label{far_comm_v2}
\zeta_i \zeta_j=\zeta_j\zeta_i
\end{equation}
for all $|i-j|\ge 2$ ({\rm far virtual commutativity}) and
\begin{equation}\label{raid3_v2}
\zeta_i \zeta_{i+1}\zeta_i=\zeta_{i+1}\zeta_i\zeta_{i+1}
\end{equation}
for $1\le i\le n-2$,
\begin{equation}
\zeta_i^2=e;
\end{equation}
\begin{equation}
\sigma_i\zeta_{i+1}\zeta_i=\zeta_{i+1}\zeta_i\sigma_{i+1};
\end{equation}
\begin{equation}
\sigma_i\zeta_j=\zeta_j\sigma_i
\end{equation}
for all $|i-j|\ge 2$.
\end{definition}

The virtual braid group will be denoted with $Br_v^n$.

\begin{lemma} \label{virt}
There exists a homomorphism from the the group $Br_2^n$ into the virtual braid group $Br_v^n$.
\end{lemma}
\begin{proof}
Consider the following mapping from the group of $\mathbb{Z}_2-$braids to virtual braids:
$$\sigma_{i,0}\to\sigma_i, \quad \sigma_{i,1}\to\zeta_i.$$
In other words, every odd crossing is sent to a virtual crossing.

An elementary check shows that (due to the parity properties) the mapping is well defined: if two elements of the group $Br_2^n$ are equal, their images are also equal as the elements of the group $Br_v^n$.
\end{proof}

\begin{remark}
Note, that this mapping doesn't provide us with a homomorphic {\it inclusion} of virtual braids into $\mathbb{Z}_2-$braids. The inverse mapping, sending $\sigma_i$ to $\sigma_{i,0}$ and $\zeta_i$ to $\sigma_{i,1}$ is ``bad'' since the image of the third virtual Reidemeister moves (with all crossings being virtual) in $\mathbb{Z}_2-$braids is not an allowed move and a square of an odd generator is not trivial, while $\zeta_i^2=e$ in virtual braids.
\end{remark}

Let us now prove Theorem \ref{classic_in_2}.
\begin{proof}
Consider two classical braids $\gamma_1, \gamma_2 \in Br^n$ such that their images in $Br_2^n$ $\tilde{\gamma}_1, \tilde{\gamma}_2$ are equal. Let us show that the braids $\gamma_1, \gamma_2$ are equal as classical braids.

Since the braids $\tilde{\gamma}_1, \tilde{\gamma}_2$ are equal, they are connected by a chain of moves: second and third Reidemeister moves with certain conditions on the parity of the crossings involved. Due to Lemma \ref{virt} their images in the virtual braid group are also equal. But it is well known (see for example \cite{braids}) that if two classical braids are equal as virtual, they are equal as classical. Therefore $\gamma_1, \gamma_2$ are equal in the class of classical braids and the proof is concluded.
\end{proof}

\section{$G-$braids}
The idea of endowing a crossing with an additional structure can be generalized in the following manner. Let $G$ be an arbitrary group. Let us define the $G-$braid group.

\begin{definition}[$G-$braid group]
The $n-$strand $G-$braid group is a group with generators $\sigma_{i,g},$ $i=1,\dots, n-1, \, g \in G$ and the following set of relations:
\begin{equation}\label{far_comm_G}
\sigma_{i,g} \sigma_{j,h}=\sigma_{j,h}\sigma_{i,g}
\end{equation}
for all $|i-j|\ge 2$ and all $g, h \in G$ ({\rm far commutativity}) and
\begin{equation}\label{raid3_G}
\sigma_{i,g} \sigma_{i+1,h}\sigma_{i,w}=\sigma_{i+1,w^{-1}}\sigma_{i,h^{-1}}\sigma_{i+1,g^{-1}}
\end{equation}
for $1\le i\le n-2$ and $ghw = 1_G$; here $1_G$ means the neutral element in the group $G$ (see Fig. \ref{raid3_g}).
\end{definition}

\begin{figure}
\begin{center}
\includegraphics[scale=0.9]{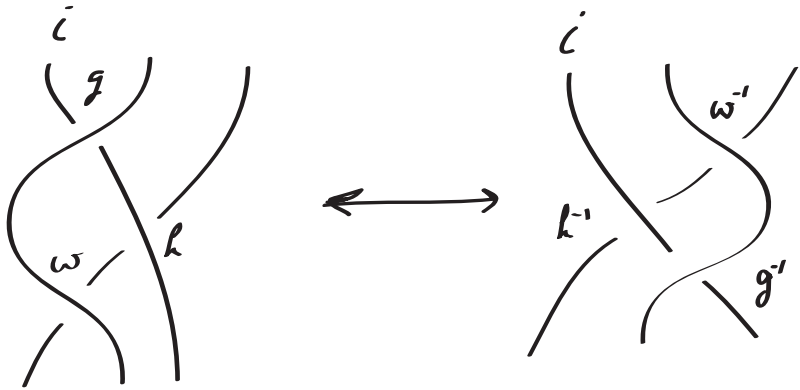}
\caption{Third Reidemeister move for a $G-$braid.}
 \label{raid3_g}
\end{center}
\end{figure}

We shall say that a braid diagram {\it admits a source--sink structure} if there exists such an orientation of its graph edges that for every four--valent vertex two opposite edges are incoming and two are outcoming. 
From the geometric point of view the $G-$braid structure means the decoration of every crossing with an element of the group $G$ with the condition that the product of the elements in every lune and triangle, written according to the source--sink structure, equals the neutral element in $G$.

\begin{example}
As an example consider the group $\mathbb{Z}_2$. In that case the inverse for any element is the element itself, so for every three labels $\varepsilon, \lambda, \nu \in \mathbb{Z}_2$ corresponding to to the crossings in the third Reidemeister move, we have $\varepsilon + \lambda + \nu \equiv 0 \mod 2$. In other words, exactly two or zero of those labels are all equal to 1.

Therefore we see, that the group $Br_{\mathbb{Z}_2}^n$ is isomorphic $Br_2^n$.
\end{example}

Note, that in case of, say, virtual braids the source--sink structure is not always admitted. Therefore the definition of $G-$braids can't be given verbatim and requires additional considerations. Nevertheless, it is possible to define analogous constructions in the virtual case, not unlike the $G-$knots and virtual $G-$knots (see, for example, \cite{gknots}). It is important to remember, though, that such object requires caution in the definition of what products of group elements, attributed to the crossings, equals the group neutral element.

\section{Dotted braids}

The parity structure, introduced in the crossings, can be reimagined as a parity on edges. This approach broadens the class of considered objects: given a diagram with a parity in the crossings one can define a diagram with a parity (dots) on the edges. The inverse procedure is not alway well defined. Nevertheless, the class of dotted diagrams such that there exists a corresponding diagram with parity, is invariant under the moves; moreover, the moves on the dotted diagrams allow the recovery of the moves on the diagrams with a parity.

\begin{definition}
The $n-$strand dotted braid group is a group with the generators $\sigma_1, \dots, \sigma_{n-1}$, generators $\gamma_1, \dots, \gamma_n$ and the following set of relations:
\begin{equation}\label{far_comm_dot}
\sigma_i \sigma_j=\sigma_j\sigma_i
\end{equation}
for all $|i-j|\ge 2$ ({\rm far commutativity}) and
\begin{equation}\label{raid3_dot}
\sigma_i \sigma_{i+1}\sigma_i=\sigma_{i+1}\sigma_i\sigma_{i+1}
\end{equation}
for $1\le i\le n-2$;
\begin{equation}\label{dot_square}
\gamma_i^2=e
\end{equation}
for all $i=1,\dots, n$,
\begin{equation}\label{dot_commute}
\gamma_i\gamma_j=\gamma_j\gamma_i
\end{equation}
for all $i,j=1,\dots, n$,
\begin{equation}\label{four_dots}
\gamma_i\gamma_{i+1}\sigma_i\gamma_i\gamma_{i+1}=\sigma_i
\end{equation}
for all $i=1,\dots, n-2$.
\end{definition}

This construction can be visualized as follows. The generators $\sigma_i$ have the same geometric meaning as before. The generators $\gamma_i$ can be represented as a trivial braid on $n$ strands with a dot on the $i-$th strand. The relation (\ref{dot_square}) means that two consecutive dots cancel each other; dots on different strands commute (see the relation (\ref{dot_commute})); finally, a crossing with four dots around equals a crossing without dots (see the relation (\ref{four_dots})).

Denote the dotted braid group with $Br_d^n$. This group is quite large. For example, one has the following

\begin{theorem}
The group $Br_2^n$ can be homomorphically included into the group $Br_d^n$.
\end{theorem}
\begin{proof} Consider the following mapping $f\colon Br_2(n)\to Br_d(n):$
$$f(\sigma_{i,0})=\sigma_i, \quad f(\sigma_{i,1})=\gamma_i\sigma_i\gamma_{i+1}.$$
This mapping is well defined and is a homomorphism to a subgroup of $Br_d^n$. It is well defined due to the fact that, thanks to the (\ref{four_dots}) relation, any lune or triangle, considered in the second or the third Reidemeister move, can be transformed into a lune or triangle without dots for which a usual Reidemeister move can be used; after that the dots are returned in the same manner (see Fig. \ref{o2}, \ref{o3}). Here we extensively use the structure of the group $Br_2(n)$.

\begin{figure}
\begin{center}
\includegraphics[scale=0.8]{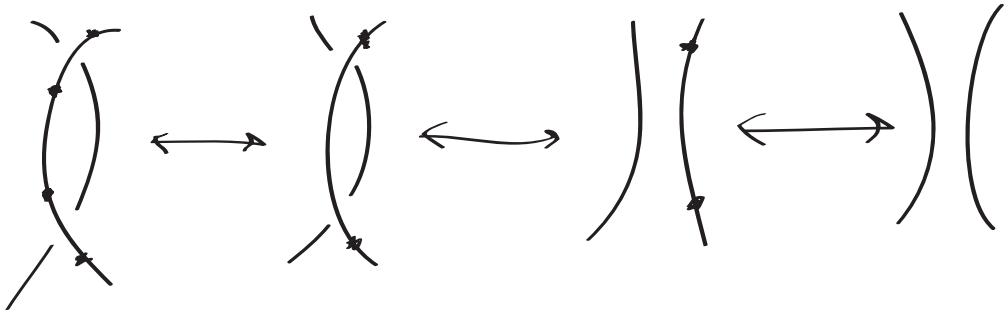}
\caption{The second Reidemeister move for the images of two odd crossings.}
 \label{o2}
\end{center}
\end{figure}

Now we need to introduce the notion of {\it good} dotted braids.

\begin{definition}
A dotted braid is called {\it good} if its every strand has even number of dots.
\end{definition}

Now we can define the inverse homomorphism: from good dotted braids to braids with parity. Consider a good braid $\gamma \in Br_d^n$ and its arbitrary diagram. We present a well defined way to assign a parity to every crossing of the diagram.

\begin{figure}
\begin{center}
\includegraphics[scale=0.9]{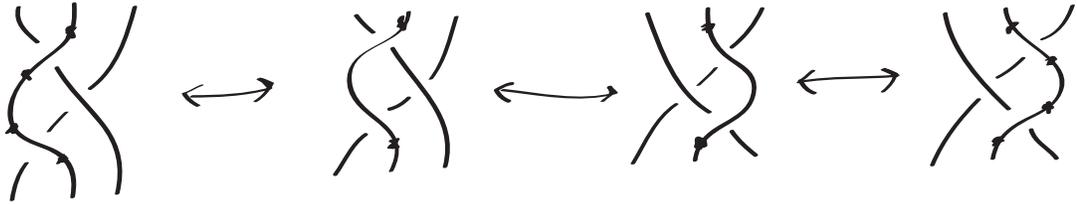}
\caption{The third Reidemeister move for the images of two odd crossings.}
 \label{o3}
\end{center}
\end{figure}

Every crossings breaks two strands into four ``half--strands''. A {\it half} is two right (or two left) half--strands together. Note, that since the braid is good, the parity of the number of dots on either halves is the same. That is the parity we attribute to the crossing considered. It is easily verified, that this mapping is well defined.

That completes the proof.\end{proof}

Geometrically speaking, we send odd crossings to the crossings with two dots on the different halves of the upper strand. Note, that such parity is a global property: a dot can be ``taken far away'' via the dot moves, but the crossing remains odd. That is exactly the difference between dotted braids and, say, $G-$braids: in the latter the additional information is positioned in the crossings hence local.

The following important fact holds true:

\begin{theorem}
If two $\mathbb{Z}_2-$braids are equal as dotted braids, they are equal as $\mathbb{Z}_2-$braids.
\end{theorem}
\begin{proof}
Consider two $\mathbb{Z}_2-$braids $\Gamma_1, \Gamma_2$ equal as dotted braids: their images under the mapping $f$, defined above, can be connected by a chain of moves allowed for the group $Br_d^n$.

The braid $f(\Gamma_1)$ is a good dotted braid. As was seen above, dotted braid moves on good braids don't change the parity of its crossings. Therefore, at each step of the transformation the braid remains good hence --- a braid with parity. Moreover, the $\mathbb{Z}_2-$images are also connected by the corresponding move. Therefore, the theorem is proved.
\end{proof}

One can use the dot idea to get different interesting groups. Let us present another example.

In virtual braids, unlike the classical ones, there is a torsion. To approach realisation of virtual braids with classical ones with an additional structure we introduce the notion of {\it twisted dotted braid group}:

\begin{definition}
The $n-$strand twisted dotted braid group is a group with the generators $\sigma_1, \dots, \sigma_{n-1}$, generators $\gamma_1, \dots, \gamma_n$ and the following set of relations:
\begin{equation}\label{far_comm_dot}
\sigma_i \sigma_j=\sigma_j\sigma_i
\end{equation}
for all $|i-j|\ge 2$ ({\rm far commutativity}) and
\begin{equation}\label{raid3_dot}
\sigma_i \sigma_{i+1}\sigma_i=\sigma_{i+1}\sigma_i\sigma_{i+1}
\end{equation}
for $1\le i\le n-2$;
\begin{equation}\label{dot_square'}
\gamma_i^2=e
\end{equation}
for all $i=1,\dots, n$,
\begin{equation}\label{dot_commute'}
\gamma_i\gamma_j=\gamma_j\gamma_i
\end{equation}
for all $i,j=1,\dots, n$,
\begin{equation}\label{four_dots'}
\gamma_i\gamma_{i+1}\sigma_i\gamma_i\gamma_{i+1}=\sigma_i^{-1}
\end{equation}
for all $i=1,\dots, n-2$.
\end{definition}

This group differs from the dotted braid group only in the last relation: four dots not preserve but invert a crossing. This group will be denoted with $\widetilde{Br}_{d}^n$. The following is true:

\begin{theorem}
The group $Br_2^n/\langle\sigma_{i,1}=\sigma_{i,1}^{-1} \rangle$ homomorphically includes into the group $\widetilde{Br}_{d}^n$.
\end{theorem}
\begin{proof}
Consider the mapping, which works on the generators exactly like the above--defined mapping $f$. Let us show that it is well defined. The triangle (third Reidemeister move) case is considered exactly as before. But here we need to prove that the following equality holds:
$$f(\sigma_{i,1})f(\sigma_{i,1})=e.$$
In other words, a lune with two odd crossings of the same structure can be acted upon with the second Reidemeister move. That is true due to the ``twisting'' relation (\ref{four_dots'}), see Fig. \ref{o2'}. The theorem is proved.
\end{proof}

\begin{figure}
\begin{center}
\includegraphics[scale=0.9]{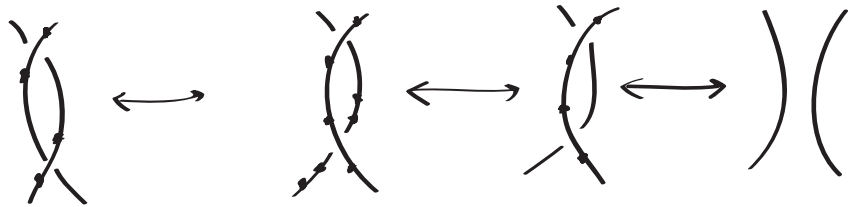}
\caption{The second Reidemeister move for the images of two odd crossings with the same structure.}
 \label{o2'}
\end{center}
\end{figure}

\begin{remark}
The group $Br_2^n/\langle\sigma_{i,1}=\sigma_{i,1}^{-1} \rangle$ ``almost'' realises the virtual braid group by sending even crossings to classical and odd to virtual. The only move absent in this group is the third Reidemeister move with three virtual crossings.
\end{remark}

\section{Further developments and some open problems}

Let us conclude the paper with a list of some open problems in braid theory, closely related to the above--defined ideas and structures.

\begin{enumerate}
\item The objects with dots, odd crossings and other such structures (knots, links, etc.) admit natural operations of {\it projection, covering, bracket, cobracket (delta)} (see, say, \cite{reidmoves}). For example, the knots whose vertices are naturally marked by elements from $\mathbb{Z}_{2}$ can be naturally covered with two component links: even crossings correspond to selfcrossings of a component and the odd ones --- to the crossings of different link components.

A question arises: is there an inclusion of the above--defined groups ($Br_2^n$, $Br_d^n$ and others, as well as their analogs) into the classical braid group $Br^N$ for a sufficiently large number of strands $N$? If that is the case, the problem of conjugation can be solved, exact linear presentations of virtual braid groups can be constructed.

\item An approach similar to free knot/free braid approach can be used when studying $G_{n}^{k}$
braids for arbitrary $k$.The structures not unlike the ones presented in the paper may be introduced to $G_n^2-$tangles and $G_n^3-$tangles. A natural problem is to explore the properties of $G_n^2-$tangles and $G_n^3-$tangles with parity, dots, $G$ group structure. That can help to understand new properties and construct new invariants for different well--known objects, for example, classical knots.
\end{enumerate}

\section*{Acknowledgments}
The authors are partially
supported by Laboratory of Quantum Topology of Chelyabinsk State
University (Russian Federation government grant 14.Z50.31.0020), by
RF President NSh --- 1410.2012.1,
 and
by grants of the Russian Foundation for Basic Research,
13-01-00830, 14-01-91161, 14-01-31288.
Zhiyun Cheng is supported by NSFC 11301028 and the Fundamental Research Funds for Central Universities of China 105-105580GK.



\begin{thebibliography}{0}
\bibitem{parity}{Manturov, V.O., Parity in knot theory. // Sbornik Mathematics, 201(5), pp. 693--733. DOI 10.1070/SM2010v201n05ABEH004089.}

\bibitem{reidmoves}{Manturov V.O., Reidemeister moves and groups. // e-print arXiv:1412.8691.}

\bibitem{bracket}{Manturov V.O., One--Term Parity Bracket For Braids. // e-print arXiv:1501.00580.}

\bibitem{braids}{Manturov V.O., An elementary proof that classical braids embed into virtual braids. // e-print arXiv:1504.03127.}

\bibitem{nonreid}{V.O.Manturov, Non--Reidemeister Knot Theory and Its Applications in Dynamical Systems, Geometry, and Topology, e-print arXiv:1501.05208}

\bibitem{gknots}{Fedoseev D.A., Manturov V.O., Invariants of homotopic classes of curves and graphs on 2--surfaces // Fund. and Appl. Mathem., 2013, v. 18, issue 4, pp. 89-105.}

\bibitem{imn}{Ilyutko D.P., Manturov V.O., Nikonov I.M., Parity in Knot Theory and Graph-Links // CMFD, 41 (2011), pp. 3--163}

\bibitem{fenn}{Fenn R.A., Rimanyi P. and Rourke C.P., The Braid-Permutation Group // Topology, vol. 36, \No. 1, pp. 123--135, 1997}
\end{thebibliography}
\end{document}